\documentclass[11pt]{article}
\usepackage{epsfig}
\usepackage{amssymb}
\usepackage{amsmath}

\usepackage{graphicx}

\usepackage{setspace}
\usepackage{enumitem}
\setlist{nolistsep}

\usepackage{lscape}

\newcommand{\Label}{\label}

\newcommand{\PG}{\textup{PG}}

\newcommand{\Q}{\mathcal{Q}}
\newcommand{\T}{\mathcal{T}}
\renewcommand{\H}{\mathcal{H}}

\newcommand{\K}{\mathcal{K}}

\newcommand{\E}{\mathcal{E}}

\renewcommand{\S}{\mathcal{S}}

\newcommand{\h}{{\mathsf h}}

\newcommand{\sw}{s}
\renewcommand{\sb}{t}

\newtheorem{theorem}{Theorem}[section]
\newtheorem{result}[theorem]{Result}
\newtheorem{lemma}[theorem]{Lemma}

\newtheorem{remark}[theorem]{Remark}

\newenvironment{proof}{\noindent{\bfseries Proof}\hspace{0.5em}}{ \null \hfill $\square$ \par}

\usepackage{fancyhdr}
\usepackage[usenames,dvipsnames]{color}

\begin{document}

\title{Characterising hyperbolic hyperplanes of a non-singular quadric in $\PG(4,q)$}
\author{S.G. Barwick, Alice M.W. Hui, Wen-Ai Jackson and Jeroen Schillewaert
\date{\today}
}

\maketitle



\begin{abstract}
Let $\H$ be a non-empty set of hyperplanes in $\PG(4,q)$, $q$ even,  such that every point of $\PG(4,q)$ lies in either $0$, $\frac12q^3$ or $\frac12(q^3+q^2)$ hyperplanes of $\H$, and every plane of $\PG(4,q)$ lies in $0$ or at least $\frac12q$ hyperplanes of $\H$. Then $\H$ is the set of all hyperplanes which meet a given non-singular quadric $Q(4,q)$ in a hyperbolic quadric.
\end{abstract}

\section{Introduction}

This article gives a characterisation of the hyperbolic hyperplanes of a  non-singular quadric  $Q(4,q)$ of $\PG(4,q)$, see \cite{HT}  for background on non-singular quadrics. Note that we refer to the hyperplanes of $\PG(4,q)$ as solids.
 There are a number of known characterisations of points, lines, planes and solids relating to the non-singular quadric $Q(4,q)$.  
 In 1956, Tallini \cite{tallini} characterised sets of points in $\PG(n,q)$ by their intersection numbers with lines, the following is a version specialised to $\PG(4,q)$.

\begin{result}\label{tallini}\cite{tallini} Let $\mathcal A$ be a set of points in $\PG(4,q)$,  $q \neq 2$, with $ |\mathcal A| \geq q^{3}+q^2+q+1 $. Suppose every line of $\PG(4,q)$ contains $0$, $1$, $2$ or $q+1$ points of $\mathcal A$. Then $\mathcal A$ is either all of $\PG(4 ,q)$; a quadric; the union of a quadric and its nucleus ($q$  even); or the union of a solid and a subspace.
\end{result}

Buekenhout \cite{buekenhout} characterised a set of points in $\PG(n,q)$ by looking at its plane intersections; and a modification specific to $\PG(4,q)$ was given by Butler in \cite{butler1}.
%
%
Ferri and Tallini \cite{ferri} proved a characterisation of the pointset of $Q(4,q)$ using plane and solid intersections.
%
%
%
The fourth author proved a corollary of this result in \cite{schil} (see below), which we will use in our proof. In joint work with De Winter \cite{DWS}, he characterised non-singular polar spaces by their intersection numbers with solids and codimension 2 spaces, which in particular reproves the following result   without using Tallini's result.
It is not possible to characterise quadrics by solid intersection numbers alone due to the existence of quasi-quadrics \cite{QQ}.

\begin{result} \label{schil}\cite{schil}
Let ${\mathcal K}$ be a set of points in $\PG(4,q)$, such that each plane of $\PG(4,q)$ meets $\K$ in either $1$, $q+1$ or $2q+1$ points, and each solid meets $\K$ in either $q^2+1$, $q^2+q+1$ or $(q+1)^2$ points, then $\K$ is the set of points of
 a non-singular quadric $Q(4,q)$.
\end{result}

When $q$ is odd, there is a  polarity associated with   $Q(4,q)$ which interchanges planes and lines. Using this polarity, we can dualise these   characterisations  when  is  $q$ odd. For example, dualising Results~\ref{tallini} and \ref{schil} give  characterisations of the tangent solids of a non-singular quadric $Q(4,q)$,  $q$ odd. 
 As a counterpoint to these, this article gives a characterisation involving solids when $q$ is even.

There are also characterisations known of sets of lines and planes related to $Q(4,q)$. 
For  $q$  odd, \cite{deresmini}, \cite{bichara} and \cite{butler1} give 
 characterisations of the  set of tangents and generator lines of  $Q(4,q)$.
These results can be dualised  using the polarity of $Q(4,q)$, $q$ odd,  to give   characterisations of a set of planes   that meet  $Q(4,q)$ in a conic.
The following result of 
Butler \cite{butler1} gives a 
combinatorial   characterisation of a set of planes  in $\PG(4,q)$ that meet  $Q(4,q)$ in a conic which holds for all $q$.  

\begin{result}\Label{Butler2}\cite{butler1} Let $\S$ be a set of planes in $\PG(4,q)$, such that
\begin{itemize}
\item[(I)]  every point lies on $q^4$ or $q^4 +q^2$ planes of $\S$;
 \item[(II)] every line lies on $0$ or $q^2$ planes of $\S$;
 \item[(III)] every solid contains at least one plane of $\S$.
  \end{itemize}
Then $\S$ is the set of planes meeting a non-singular quadric $Q(4,q)$ in a conic.
\end{result}


This article 
 proves  a   combinatorial characterisation of solids  in the style of  Result~\ref{Butler2}. 
We first count   hyperbolic  solids  incidences in $Q(4,q)$.

\begin{lemma}
Let $Q(4,q)$ be a non-singular quadric in $\PG(4,q)$, $q$ even, and let  $\H$ be the set of solids  that meet $Q(4,q)$ in a hyperbolic quadric. Then 
\begin{itemize}
\item[(a)]  every point of $\PG(4,q)$ lies on $0$ or $\frac12q^3$ or $\frac12(q^3+q^2)$  solids  of $\H$;

  \item[(b)]  every line of $\PG(4,q)$  lies in  $0$, $\frac12q(q-1)$, $\frac12q^2$, $\frac12q(q+1)$ or $q^2$   solids   of $\H$; and
  \item[(c)]  every plane of $\PG(4,q)$  lies in    $0$, $\frac12q$  or $q$   solids   of $\H$.  \end{itemize}
\end{lemma}

\begin{proof}
All references in this proof refer to \cite{HT}.  By Corollary 1.8, $Q(4,q)$ has a nucleus. By Lemma 1.13(ii) and Theorem 1.54, a solid is a tangent solid if and only if it contains the nucleus. Moreover Theorem 1.57 counts how many hyperbolic, elliptic and tangent solids there are. By Theorem 1.46 the subgroup of $\mathrm{PGL}(5,q)$ which fixes $Q(4,q)$ acts transitively on the set of points on $Q(4,q)$ and by Theorem 1.49 it has two orbits on the points of $\PG(4,q)\setminus Q(4,q)$, one of which is a singleton, namely the nucleus. Combining the above facts with the exact numbers from \cite{HT} it follows that the nucleus is not contained in any hyperbolic solid, the other points not in $Q(4,q)$ are contained in exactly $\frac12q^3$ hyperbolic solids, and the points of $Q(4,q)$ are contained in exactly $\frac12(q^3+q^2)$ hyperbolic solids, proving part (a).

Part (c) can be proved similarly by counting the incident plane-solid pairs $(\pi, \Pi)$ where $\pi \subset \Pi$, (where $\Pi$ is a hyperbolic solid) for different types of planes $\pi$ using Theorem 1.57, Lemma 1.13 and Corollary 1.68, see also Example 1.70. More precisely, any plane not containing the nucleus which meets $Q(4,q)$ in a conic is contained in exactly
$\frac12q$ hyperbolic solids, any plane meeting $Q(4,q)$ in 2 lines is contained in exactly $q$ hyperbolic solids, all other planes are not contained in any hyperbolic solids. Part (b) can be proved similarly by counting   incident line-solid pairs. 
\end{proof}

%
%
%
%
%

In this article, we assume part (a) and a weaker version of part (c) to prove a characterisation of the hyperbolic solids of   $Q(4,q)$, $q$ even. The main result of this article is the following. 

\begin{theorem}\label{mainresult} Let $\H$ be a non-empty set of solids in $\PG(4,q)$, $q$ even,  such that
 \begin{itemize}
\item[(I)] every point of $\PG(4,q)$ lies in either $0$, $\frac12q^3$ or $\frac12(q^3+q^2)$ solids of $\H$, and
\item[(II)]
a plane which is contained in a solid of $\mathcal{H}$ is contained in at least $\frac12q$ solids of $\mathcal{H}$.
\end{itemize}
Then the set of points contained in $\frac12(q^3+q^2)$ solids of $\H$ forms a non-singular quadric $Q=Q(4,q)$ and
$\H$ is the set of solids that meets $Q$ in a non-singular hyperbolic quadric $Q^+(3,q)$.
\end{theorem}

Consider a quasi-quadric as defined in \cite{QQ}, more precisely: Let $n = 2m$ and let $q$ be even. A \emph{parabolic quasi-quadric} with nucleus the point $N$ in $\mathrm{PG}(n, q)$ is a set $\mathcal{K}$ of 
$\frac{q^n-1}{q-1}$ points such that each line on $N$ contains a unique point of $\mathcal{K}$ and each hyperplane not on $N$ meets $\mathcal{K}$ in $\frac{(q^m+1)(q^{m-1}-1)}{q-1}$ or 
$\frac{(q^m - 1)(q^{m-1} + 1)}{q - 1}$ points. (Note for $n=4$ these values are $q^2+1$ and $(q+1)^2$). These are constructed in Theorem 10 of \cite{QQ} using a pivoting method and for all $q>2$ even there exist quasi-quadrics which are not quadrics.

An auxiliary result leading to Theorem \ref{mainresult} is 

\begin{theorem}\label{auxiliary}
Let $\H$ be a non-empty set of solids in $\PG(4,q)$, $q$ even, satisfying condition (I).
Then $\H$ is the set of solids intersecting a quasi-quadric in $\mathrm{PG}(4,q)$ (which is formed by the points contained in $\frac{1}{2}(q^3+q^2)$ elements of $\H$) in $q^2+2q+1$ points.
\end{theorem}

We show the necessity of Condition (II) in Theorem \ref{mainresult} by showing:
\begin{lemma}
The set of solids intersecting a quasi-quadric $\mathcal{Q}$ in $\mathrm{PG}(4,q)$ in $(q+1)^2$ points satisfies Condition (I). 
In particular,  every quasi-quadric in $\mathrm{PG}(4,2)$ is a non-singular quadric.
\end{lemma}
\begin{proof}
Let $\H$ be the set of hyperplanes which intersect $\mathcal{Q}$ in $(q+1)^2$ points and let $N$ denote the nucleus. Fix a point $X\in \mathcal{Q}$ and let $H$ be the number of elements of $\mathcal{H}$ containing $X$. We count in two ways the pairs $(P,\Pi)$ where $P\in \mathcal{Q}, P\neq X$ where $\Pi$ is a solid containing $P,X$ but not $N$. Firstly we count the solids through $X$, and then we look at the points $P\in\Q$, $P\ne X$.  Note that from the definition of the quasi-quadric the line $PX$ does not contain $N$, so  there are $q^2$ solids through $PX$ which do not contain $N$.
\[
H(q^2+2q) + (q^3-H)q^2 = (q^3+q^2+q)q^2.
\] 
We use the same argument for a point $Y$ not belonging to the quadric and different from $N$ and let $T$ be the number of elements of $\mathcal{H}$ containing $Y$.  Note that the line $NY$ has one point $Z$ of $\Q$ and so in this case there are no solids in $\H$ containing $Y$ and $Z$.
\[
T(q^2+2q+1) + (q^3-T)(q^2+1) = (q^3+q^2+q)q^2+1.0.
\]
Hence $H = \frac12(q^3+q^2)$ and $T = \frac{1}{2}q^3$.
\end{proof}

\section{Proof of Theorem~\ref{mainresult}}

Let $\H$ be a non-empty set of solids in $\PG(4,q)$, $q$ even,   such that
 \begin{itemize}
\item[(I)] every point of $\PG(4,q)$ lies on $0$ or $\frac12q^3$ or $\frac12(q^3+q^2)$ solids of $\H$, and \item[(II)] a plane which is contained in a solid of $\mathcal{H}$   is contained in at least $\frac12q$ solids of $\mathcal{H}$.

\end{itemize}

By assumption (I), there are three types of points. A point is called a
\begin{itemize}
\item \emph{red point} if it lies in $0$ solids of $\H$,
\item \emph{white point} if it lies in $\frac12q^3$ solids of $\H$,
\item \emph{black point} if it lies in $\frac12(q^3+q^2)$ solids of $\H$.
\end{itemize}

Throughout, we will denote the number of red points of $\PG(4,q)$ by $r$, the number of white points of $\PG(4,q)$ by $w$, and the number of black points of $\PG(4,q)$ by $b$.

A solid in $\H$ does not contain any red points, and we partition the solids of $\PG(4,q)$ into $\H$ and two further types of solids:
\begin{itemize}
\item
let $\T$ be the set of solids containing at least one red point,
\item
let $\E$ be the set of solids not in $\H$ that contain no red points.
\end{itemize}

In the following, we will show that the black points form a non-singular quadric $Q(4,q)$ whose nucleus is the unique red point. Further, we show that $\H$ is the set of solids meeting $Q(4,q)$ in a hyperbolic quadric. Indeed, $\E$ is the set of solids meeting $Q(4,q)$ in an elliptic quadric, and $\T$ is the set of solids meeting $Q(4,q)$ in a conic cone.

\smallskip

We proceed with a series of lemmas.

\begin{lemma}\Label{lem:A-size1}
Let  $$\h= \dfrac{|\H|}{\frac12q^2},$$  then $\h$ is an integer congruent to either 0 or 1 modulo $q$,    and $q+1$ divides $\h(\h-2)$.
\end{lemma}

\begin{proof}
We first show that $\frac12q^2$ divides $|\H|$.
Count in two ways the incident pairs $(P,\Pi)$ where $\Pi$ is a solid in $ \H$ and $P$ is a point of $\Pi$.
\begin{eqnarray}
r.0+w.\tfrac12 q^3 + b.\tfrac12(q^3+q^2)& =& |\H|(q^3+q^2+q+1)\label{eqn:OP}
\end{eqnarray}
As $\frac12q^2$ divides the left hand side and gcd$(q^2,q^3+q^2+q+1)=1$, it follows $\frac12q^2$ divides $|\H|$.

Now we show that $\h= \dfrac{|\H|}{\frac12q^2}$ is congruent to either 0 or 1 modulo $q$.
Count in two ways incident triples $(P , \Pi,\Sigma)$ where $\Pi,\Sigma$ are distinct solids of $\H$ and $P$ is a point in $\Pi\cap\Sigma$. We have
\begin{eqnarray}
r.0.0+ w.\tfrac12q^3.(\tfrac12q^3-1)+b.\tfrac12(q^3+q^2).(\tfrac12(q^3+q^2)-1)&=&|\H|(|\H|-1)(q^2+q+1).\label{eqn:1}
\end{eqnarray}
Substituting (\ref{eqn:OP}) into (\ref{eqn:1}) to eliminate $b$ and then simplifying gives
\begin{eqnarray}\label{eqn:wq}
\begin{aligned}
\h\left( (q+1-\h)
(q^2+q+1)+q(q^3+q^2-2)\right)&=&wq.\\
\end{aligned}
\end{eqnarray}
Hence $\h (\h-1)\equiv 0 \pmod{q}$. 
As $q$ is a power of 2, it follows that either $\h \equiv 0 \pmod{q}$ or $\h \equiv 1 \pmod{q}$.
It follows from (\ref{eqn:OP}) that $q+1|w$, hence using (\ref{eqn:wq}), we deduce that  $q+1| \h(\h-2)$. 
\end{proof}

\begin{lemma}\Label{item:hsize}
Each solid in $\H$ contains $(q+1)^2$ black points, and $|\H|=\frac12 q^2(q^2+1)$.
\end{lemma}

\begin{proof}
Let $\Sigma\in\H$, let $\sw $ be the number of white points in $\Sigma$ and let $\sb $ be the number of black points in $\Sigma$. As $\Sigma$ does not contain any red points, we have
\begin{eqnarray}
\sw +\sb &=&q^3+q^2+q+1.\label{eqn:3-pt-A}
\end{eqnarray}
Count in two ways incident pairs $(P,\Pi)$ where $\Pi\in\H$, $\Sigma\ne \Pi$ and point $P\in\Sigma\cap\Pi$. We have
\begin{eqnarray*}
\sw .(\tfrac12 q^3-1)+\sb .(\tfrac12(q^3+q^2)-1)&=&(|\H|-1)(q^2+q+1).
\end{eqnarray*}
Simplifying using (\ref{eqn:3-pt-A}) using the notation of Lemma \ref{lem:A-size1} yields
\begin{eqnarray}
\sb -q&=&\h(q^2+q+1)-q^2(q^2+q+1).\label{eqn:3-pt-A-m}
\end{eqnarray}

Since $0\leq t \leq q^3+q^2+q+1$, from (\ref{eqn:3-pt-A-m}) we have $q^2 \leq \h <q^2+q$. It follows from  Lemma~\ref{lem:A-size1} that $\h=q^2+1$ and $t=(q+1)^2$ as required.
\end{proof}

We now count the number of points of each colour in $\PG(4,q)$.

\begin{lemma} \Label{item:pt-size}
There are $q^3+q^2+q+1$ black points, $q^4-1$ white points and one red point.
\end{lemma}

\begin{proof}
Count in two ways incident pairs
 $(P,\Sigma)$ where $\Sigma\in\H$ and $P$ is black point in $\Sigma$.
 By Lemma~\ref{item:hsize},
 $|\H|=\frac12q^2(q^2+1)$ and $\Sigma$ contains $(q+1)^2$ black points, so we have
\begin{eqnarray*}
b. \tfrac12(q^3+q^2)=\tfrac12q^2(q^2+1).(q+1)^2
\end{eqnarray*}
giving $b=q^3+q^2+q+1$. Substituting this into (\ref{eqn:OP}) gives $w=q^4-1$ and so $r=1$, as required.
\end{proof}

Recall that we partition the solids of $\PG(4,q)$ into three sets: $\H$, $\T$ (the solids containing the red point), and $\E$ (the solids not in $\H$ which do not contain the red point). In Lemma~\ref{item:hsize} we counted the solids in $\H$. We now do the same for $\T$ and $\E$.
\begin{lemma}\Label{item:int}
\begin{enumerate}
\item Each solid in $\E$ contains $q^2+1$ black points, and $|\E|=\frac12q^2(q^2-1)$.
\item Each solid in $\T$ contains $q^2+q+1$ black points, and $|\T|=q^3+q^2+q+1$.
\end{enumerate}
\end{lemma}

%
%
%
%

\begin{proof} By Lemma~\ref{item:pt-size} there is a unique red point, so we have $|\T|=q^3+q^2+q+1$ (the number of solids through a point). Using Lemma~\ref{item:hsize} and $|\H|+|\E|+|\T|=q^4+q^3+q^2+q+1$ gives $|\E|=\frac12q^2(q^2-1)$ as required.

Let $\Pi\in\E$, let $\sb $ be the number of black points in $\Pi$ and let $\sw $ be the number of white points in $\Pi$. As solids in $\E$ do not contain a red point, we have
\begin{eqnarray}
\sw +\sb &=&q^3+q^2+q+1.\label{eqn:Ept}
\end{eqnarray}
Count in two ways incident pairs $(P,\Sigma)$ where $\Sigma\in\H$, and $P\in\Pi\cap\Sigma$. We have
\begin{eqnarray*}
\sw .\tfrac12 q^3+\sb .\tfrac12(q^3+q^2)&=&|\H|(q^2+q+1).
\end{eqnarray*}
Simplifying using (\ref{eqn:Ept}) gives $\sb =q^2+1$
proving part 1.

For part 2, let $\Pi\in\T$, let $\sb $ be the number of black points in $\Pi$ and let $\sw $ be the number of white points in $\Pi$. As solids in $\T$ contain the unique red point, we have
\begin{eqnarray}
\sw +\sb &=&q^3+q^2+q.\label{eqn-bing}
\end{eqnarray}
 Count in two ways incident pairs $(P,\Sigma)$ where $\Sigma\in\H$, and $P\in\Pi\cap\Sigma$, noting that as $\Sigma\in\H$, it does not contain the red point. We have
 \begin{eqnarray*}
\sw .\tfrac12 q^3+\sb .\tfrac12(q^3+q^2)&=&|\H|(q^2+q+1).
\end{eqnarray*}
Simplifying using (\ref{eqn-bing}) gives $\sb =q^2+q+1$ as required.
\end{proof}

\begin{remark}\label{rem}
Notice that so far we have only used Condition (I) and to prove that the set of black points is a quasi-quadric, it remains to show that each line through the red point contains exactly one black point.
Standard counting arguments show first that every plane through the red point contains exactly $q+1$ black points and then that every line through the red point contains exactly 1 black point.


\end{remark}

\begin{lemma}\label{plane-count-new}
All planes contain $1,q+1$ or $2q+1$ black points.
\end{lemma}

\begin{proof}
Let $\pi$ be a plane, and denote the number of black points in $\pi$ by $x$.  Further, let $t$ be the number of solids of $\T $ containing  $\pi$ and $s $ the number of solids of $\H$  containing   $\pi$, so   there are $q+1-s-t$ solids of $\E$  containing   $\pi$. 
Counting the black points in    $\PG(4,q)$ by counting in solids about $\pi$, we have (using Lemmas~\ref{item:hsize},  \ref{item:pt-size} and \ref{item:int})
 $$t(q^2 +q+1)+s(q^2 +2q+1)+(q+1-s-t)(q^2 +1)=q^3 +q^2 +q+1+qx.$$
 If $\pi$ contains the red point, then $t=q+1$ and $s=0$, hence $x=q+1$. If $\pi$ does not contain the red point, then $t=1$, and so $x=2s+1$.
 If 
  $s = 0$, then $x = 1$. If $s> 0$, then using (II), we have $q/2\leq s\leq q+1$ and so   $q+1\leq x\leq 2q+1$.   
 We now show that in this case (where $\pi$ lies in a solid of $\H$) we have $x=q+1$ or $2q+1$. 

Consider a fixed solid $\Sigma\in\mathcal{H}$.
We will first count pairs $(P,\pi)$ where $P$ is a black point in a plane $\pi$ belonging to $\Sigma$ and then triples $(P,P',\pi)$ with $P,P'$ black points, $\pi\in\Sigma$ where $P\in \pi$, $P'\in \pi$, $P\neq P'$. If $\pi_i$ is a plane  belonging to $\Sigma$, let $x_i$ denote the number of black points in $\pi_i$.
We have

$$ \sum_i x_i  = (q+1)^2\frac{(q^3-1)(q^2-1)}{(q^2-1)(q-1)}=(q+1)^2(q^2+q+1),$$

$$\sum_i x_i(x_i-1) = (q+1)^2(q^2+2q)\frac{q^2-1}{q-1}=(q+1)^2(q^2+2q)(q+1).$$

Hence we obtain
$$ \sum_i x_i^2 =  (q+1)^2[q^3 + 4 q^2 + 3 q + 1].$$

So we obtain that
$$ \sum_i (x_i-(q+1))(x_i-(2q+1)) = \sum_i x_i^2-(3q+2)\sum_i x_i+(2q+1)(q+1)(q^2+1)(q+1)  = 0.$$
As
$q+1\leq x_i\leq 2q+1$, we 
have $x_i=q+1$ or $2q+1$ as required. 
\end{proof}

 {\bfseries Proof of Theorem~\ref{mainresult}}
 Let $\mathcal B$ be the set of black points.
 By Lemma~\ref{plane-count-new}, planes meet $\mathcal B$ in either $1$, $q+1$ or $2q+1$ points. By Lemmas~\ref{item:hsize} and \ref{item:int}, solids meet $\mathcal B$ in either $q^2+1$, $q^2+q+1$ or $q^2+2q+1$ points. Hence by Result~\ref{schil},
 $\mathcal B$ is the set of points of
 points of  a non-singular quadric  of $\PG(4,q)$, and so $\H$ is the set of solids of $\PG(4,q)$ that meet a non-singular quadric in a hyperbolic quadric as required. \hfill $\square$

We note that it follows that $\E$ is the set of solids that meet the non-singular quadric in an elliptic quadric, and $\T$ is the set of solids that meet the non-singular quadric in a quadratic cone.

\bigskip\bigskip

{\bfseries Author information}

S.G. Barwick. School of Mathematical Sciences, University of Adelaide, Adelaide, 5005, Australia.
susan.barwick@adelaide.edu.au

A.M.W. Hui. Statistics Program, BNU-HKBU United International College, Zhuhai, China.
alicemwhui@uic.edu.hk, huimanwa@gmail.com

W.-A. Jackson. School of Mathematical Sciences, University of Adelaide, Adelaide, 5005, Australia.
wen.jackson@adelaide.edu.au

J. Schillewaert, Department of Mathematics, University of Auckland, New Zealand.\\
j.schillewaert@auckland.ac.nz

{\bf Acknowledgements}
We would like the anonymous referee for their valuable feedback which improved both the content and exposition of this paper.
A.M.W. Hui is supported by the Young Scientists Fund (Grant No. 11701035) of the National Natural Science Foundation of China, and
J. Schillewaert by a University of Auckland FRDF grant.

\end{document}